\documentclass{amsart}

\usepackage[margin=1.4in]{geometry} 
\usepackage{amsthm}
\usepackage{amssymb}
\usepackage{mathtools}
\usepackage{float}
\usepackage[colorlinks=true,linkcolor=blue,citecolor=blue]{hyperref} 
\usepackage{zref-clever} 
\zcsetup{nameinlink,cap}
\usepackage{tikz-cd, tikz} 
\tikzset{
  commutative diagrams/.cd, 
  arrow style=tikz, 
  diagrams={>=stealth}
}
\usetikzlibrary{shapes.geometric, arrows.meta, calc, positioning, svg.path}

\usepackage[shortlabels]{enumitem} 
\usepackage[mathscr]{euscript} 
\usepackage{adjustbox}
\usepackage{bm}
\usepackage{wrapfig}

\usepackage[shortalphabetic]{amsrefs}

\newcommand{\cref}[1]{\zcref{#1}}
\newcommand{\Cref}[1]{\zcref[S]{#1}}

\zcRefTypeSetup{equation}{
  Name-sg-ab =,
  name-sg-ab =,
  Name-pl-ab =,
  name-pl-ab =,
  abbrev,
}

\numberwithin{equation}{section} 
\numberwithin{table}{section} 
\numberwithin{figure}{section} 

\theoremstyle{definition}

\AddToHook{env/definition/begin}{\zcsetup{countertype={equation=definition}}}
\newtheorem{definition}[equation]{Definition}

\AddToHook{env/remark/begin}
{\zcsetup{countertype={equation=remark}}}
\newtheorem{remark}[equation]{Remark}

\AddToHook{env/example/begin}{\zcsetup{countertype={equation=example}}}
\newtheorem{example}[equation]{Example}

\theoremstyle{plain}

\AddToHook{env/lemma/begin}{\zcsetup{countertype={equation=lemma}}}
\newtheorem{lemma}[equation]{Lemma}

\AddToHook{env/proposition/begin}{\zcsetup{countertype={equation=proposition}}}
\newtheorem{proposition}[equation]{Proposition}

\AddToHook{env/theorem/begin}{\zcsetup{countertype={equation=theorem}}}
\newtheorem{theorem}[equation]{Theorem}

\AddToHook{env/corollary/begin}{\zcsetup{countertype={equation=corollary}}}

\AddToHook{env/maintheorem/begin}{\zcsetup{countertype={equation=maintheorem}}}
\newtheorem{maintheorem}{Theorem} 
 
\zcRefTypeSetup{maintheorem}{
	Name-sg = Theorem, 
	name-sg = theorem, 
	Name-pl = Theorems, 
	name-pl = theorems
	}

\AddToHook{env/conjecture/begin}{\zcsetup{countertype={equation=conjecture}}}
\newtheorem{conjecture}[maintheorem]{Conjecture}
 
\zcRefTypeSetup{conjecture}{
	Name-sg = Conjecture, 
	name-sg = conjecture, 
	Name-pl = Conjectures, 
	name-pl = conjectures
	}

\newcommand{\into}{\hookrightarrow}
\newcommand{\onto}{\twoheadrightarrow}

\newcommand{\ul}[1]{\underline{#1}} 

\newcommand{\cA}{\mathcal{A}} 
\newcommand{\uA}{\underline{A}} 
\newcommand{\ufa}{\underline{\mathfrak{a}}} 
\newcommand{\F}{\mathbb{F}} 
\newcommand{\bbS}{\mathbb{S}} 
\newcommand{\R}{\mathbb{R}} 
\newcommand{\Z}{\mathbb{Z}} 
\newcommand{\ulZ}{\ul{\Z}}
\newcommand{\ulF}{\ul{\F_2}}
\newcommand{\ulf}{\ul{f}}
\newcommand{\upi}{\ul{\pi}}

\newcommand{\augmentationideal}{\underline{\mathfrak{c}}} 

\newcommand{\infl}[1]{\bm{ \langle } #1 \bm{ \rangle } }

\newcommand{\qfont}{\mathscr} 
\newcommand{\HQ}{\boldsymbol{\qfont{H}}}
\newcommand{\HQN}{\HQ\QN}

\newcommand{\QN}{\ul{\boldsymbol{\qfont{N}}}} 

\DeclareMathOperator{\res}{res} 
\DeclareMathOperator{\tr}{tr} 
\DeclareMathOperator{\nm}{nm} 
\DeclareMathOperator{\RO}{RO} 
\DeclareMathOperator{\MU}{MU}
\DeclareMathOperator{\MUR}{MU\mathbb{R}}

\newcommand{\orho}{\overline{\rho}} 

\newcommand{\Ab}{{\mathscr{A}\textnormal{b}}} 
\newcommand{\Fin}{{\mathscr{F}\textnormal{in}}} 
\newcommand{\Green}{{\mathscr{G}\textnormal{reen}}} 
\newcommand{\Mack}{{\mathscr{M}\textnormal{ack}}} 
\newcommand{\Tamb}{{\mathscr{T}\textnormal{amb}}} 
\newcommand{\CRing}{\mathscr{C}\textnormal{Ring}} 

\newcommand{\inductioncolor}{orange}
\newcommand{\normcolor}{blue}

\def\dbox{
     \begin{tikzpicture}
        \node at (0,0) [rectangle,draw,inner sep=0.75pt] {$\ast$};
     \end{tikzpicture}
}

\newcommand{\phiLDRf}{
\begin{tikzpicture}[scale=0.6]
\draw[fill=black] (0,0) rectangle (0.5,0.15);
\end{tikzpicture}
}

\newcommand{\fillpent}{
\begin{tikzpicture}
\node[draw, fill=black, inner sep=0.5pt, regular polygon, regular polygon sides=5, scale=0.8] at (0,0) {$\ast$};
\end{tikzpicture}
}

\newcommand{\fillpentdual}{
\begin{tikzpicture}
\node[draw, fill=black, inner sep=0.5pt,regular polygon, regular polygon sides=5, text=white, scale=0.8] at (0,0) {\scriptsize $\ast$};
\end{tikzpicture}
}

\newcommand{\filltrap}{
\begin{tikzpicture}
\node[draw, fill=black, inner sep=0.5pt, trapezium, scale=0.8] at (0,0) {$\ast$};
\end{tikzpicture}
}

\newcommand{\filltrapdual}{
\begin{tikzpicture}
\node[draw, fill=black, inner sep=0.5pt, trapezium, scale=0.8, text=white] at (0,0) {\scriptsize $\ast$};
\end{tikzpicture}
}

\newcommand{\mgw}{\ul{mgw}}

\makeatletter

\pgfdeclareshape{phiZMshape}
{%
  \inheritsavedanchors[from=circle]
  \inheritanchorborder[from=circle]%
  \inheritanchor[from=circle]{north}%
  \inheritanchor[from=circle]{north west}%
  \inheritanchor[from=circle]{north east}%
  \inheritanchor[from=circle]{center}%
  \inheritanchor[from=circle]{west}%
  \inheritanchor[from=circle]{east}%
  \inheritanchor[from=circle]{mid}%
  \inheritanchor[from=circle]{mid west}%
  \inheritanchor[from=circle]{mid east}%
  \inheritanchor[from=circle]{base}%
  \inheritanchor[from=circle]{base west}%
  \inheritanchor[from=circle]{base east}%
  \inheritanchor[from=circle]{south}%
  \anchor{south}{\centerpoint
        \pgf@xa=\radius
        \advance\pgf@y by -2.6\pgf@xa
        }
  \anchor{south west}{\pgfpointorigin \pgf@x=-\radius \pgf@y= -\radius \multiply \pgf@y by 3}
  \anchor{south east}{\pgfpointorigin \pgf@x=\radius \pgf@y= -\radius \multiply \pgf@y by 3}
  \inheritbackgroundpath[from=circle]%
  \foregroundpath{
    \centerpoint%
    \pgf@xc=\pgf@x%
    \pgf@yc=\pgf@y%
    \pgfutil@tempdima=\radius%
    \pgfmathsetlength{\pgf@xb}{\pgfkeysvalueof{/pgf/outer xsep}}%
    \pgfmathsetlength{\pgf@yb}{\pgfkeysvalueof{/pgf/outer ysep}}%
    \ifdim\pgf@xb<\pgf@yb%
      \advance\pgfutil@tempdima by-\pgf@yb%
    \else%
      \advance\pgfutil@tempdima by-\pgf@xb%
    \fi%
    \pgfpathmoveto{\pgfpointadd{\pgfqpoint{\pgf@xc}{\pgf@yc}}{\pgfqpoint{0\pgfutil@tempdima}{-1\pgfutil@tempdima}}}%
    \pgfpathlineto{\pgfpointadd{\pgfqpoint{\pgf@xc}{\pgf@yc}}{\pgfqpoint{0\pgfutil@tempdima}{-1.6\pgfutil@tempdima}}}%
    \pgfusepath{fill,stroke}%
    \pgfpathcircle{\pgfpointadd{\pgfqpoint{\pgf@xc}{\pgf@yc}}{\pgfqpoint{0pt}{-2.1\pgfutil@tempdima}}}{0.4\pgfutil@tempdima}%
    \pgfusepath{stroke}%
  }%
}%

\makeatother

\def\phiZM{
    \begin{tikzpicture}[x=1.0ex,y=1.0ex]
      \node[phiZMshape,draw,inner sep=0.4ex] at (0,0) {};
    \end{tikzpicture}
}

\def\mgwsymbol{
    \begin{tikzpicture}
      \node[fill = black, regular polygon, regular polygon sides=5,  regular polygon rotate=180, minimum width=0pt,  inner sep = 0.3ex,scale=1.7] at (0,0) {};
    \end{tikzpicture}
}

\newcommand{\sfrac}[2]{{}^{#1}\!/_{\!#2}} 

\begin{document}

\title{The Zero Slice of Quaternionic Real Bordism}
\author{Bertrand J. Guillou}
\email{bertguillou@uky.edu}
\author{Jesse Keyes}
\email{jdke228@uky.edu}
\author{David Mehrle}
\email{davidm@uky.edu}
\address{Department of Mathematics, University of Kentucky, Lexington, KY, U.S.A.}


\begin{abstract}
Using the Hill-Hopkins-Ravenel norm, one can produce a $Q_8$-spectrum $N_{C_2}^{Q_8} \MUR$. Working towards a computation of the slice spectral sequence for $N_{C_2}^{Q_8} \MUR$, we compute the zero slice of $N_{C_2}^{Q_8} \MUR$ and a bigraded subring of the $\RO(G)$-graded homotopy Mackey functors of this slice. \end{abstract}

\maketitle

\setcounter{tocdepth}{1}
\tableofcontents

\section{Introduction}
\label{Intro}

One of the guiding problems of stable homotopy theory is to compute the stable homotopy groups of spheres, also known as the homotopy groups of the sphere spectrum $\bbS  \coloneqq \Sigma^\infty S^0$.
Chromatic homotopy theory suggests to understand this via a 
tower
\[
	\bbS \simeq \mathrm{holim}\left(  \cdots \to L_{E(n)}\bbS \to L_{E(n-1)}\bbS \to \cdots \to L_{E(0)}\bbS\right),
\]
where $L_{E(n)} \bbS$ denotes the Bousfield localization of $\bbS$ with respect to the Johnson-Wilson theory $E(n)$. The difference between successive stages in this filtration is controlled by $L_{K(n)}\bbS$, the localization with respect to height $n$ Morava $K$-theory $K(n)$.

Thus the spectra $L_{K(n)}\bbS$ have become central objects of study. These are rich objects, whose rational homotopy groups have only recently been computed \cite{BSSW}.
One means of studying the $K(n)$-local sphere comes from the identification of $L_{K(n)}\bbS$ with the homotopy fixed points of the action of the Morava stabilizer group $\mathbb{G}_n$ on the Lubin-Tate spectrum $E_n$ \cite{DH}.

Given this identification, we may study $L_{K(n)} \bbS$ by means of the corresponding homotopy fixed points spectral sequence, which has signature 
\[ 
	E_2 = H^*_{\text{cts}}(\mathbb{G}_n; \pi_*E_n) \Rightarrow \pi_*L_{K(n)}\bbS. 
\]
Computing the $E_2$-page of this spectral sequence requires wrestling with the mysterious action of $\mathbb{G}_n$ on $E_n$.

To sidestep the action of $\mathbb{G}_n$ on $E_n$, one can restrict to finite subgroups $H \leq \mathbb{G}_n$ and approximate $L_{K(n)} \bbS \simeq E_n^{h\mathbb{G}_n}$ by $E_n^{hH}$. 
A theorem of  Hewett can be applied to obtain a complete description of the maximal 
finite subgroups $H$ of
$\mathbb{G}_n$ in terms of height \cite{Hewett} (see also \cite{Bu}). 

At heights not congruent to two modulo four, the maximal $2$-subgroup of $\mathbb{G}_n$ is a cyclic $2$-group $C_{2^k}$. Among other reasons, this has encouraged the study of equivariant lifts of key players in chromatic homotopy theory when the group of equivariance is a cyclic $2$-group. For example,  the landmark resolution of the Kervaire Problem by Hill-Hopkins-Ravenel was given by analyzing the slice spectral sequence of $N_{C_2}^{C_{2^k}} \MUR$ \cite{HHR}. Here, $\MUR$ is a $C_2$-equivariant lift of the complex bordism spectrum $\MU$ and $N_{C_2}^{C_{2^k}}(-)$ is the multiplicative norm construction for equivariant spectra. 

At heights congruent to two modulo four, the maximal 2-subgroup of $\mathbb{G}_n$ is given by $Q_8$, the quaternion group. Quaternion-equivariant homotopy theory is less well-studied than $C_{2^k}$-equivariant homotopy theory, but we expect that it is no less important to the project of chromatic homotopy theory. In particular, we would like to understand the slice spectral sequence of the quaternionic norm of real bordism $N_{C_2}^{Q_8} \MUR$.

Computing the slice spectral sequence for an arbitrary equivariant spectrum $X$ requires understanding the slices of $X$ and the homotopy of these slices.
The simplest case is the zero slice of $X$, which is the equivariant Eilenberg--Mac~Lane spectrum $H\upi_0 X$.
When $X \simeq N_H^G Y$ with $Y$ connective, Ullman gives an identification \cite{U}*{Corollary~3.2}
\[
	\upi_0 N_H^G Y \cong N_H^G \upi_0 Y,
\]
where the norm in the right side of the equation is the norm for Mackey functors \cites{HM,Ho,U}.

\begin{wrapfigure}{r}{0.465\textwidth}
	\vspace*{-0.5cm}
		\begin{tikzpicture}
			\clip (-3.5,-0.2) rectangle (3.5,6.5);
			\node(Q) at (0,6) {$\Z \oplus \F_2^2$};
			\node(I) at (-3,4) {$\Z$};
			\node(J) at (0,4) {$\Z$};
			\node(K) at (3,4) {$\Z$};
			\node(C) at (0,2) {$\Z$};
			\node(e) at (0,0) {$\Z$};

			\draw[->>,bend right=10] (Q) to node[above left,scale=0.5] {$\begin{bmatrix} 1 & 0 & 0 \end{bmatrix}$} (I);
			\draw[->>,bend right=10] (Q) to node[left, scale=0.5] {$\begin{bmatrix} 1 & 0 & 0 \end{bmatrix}$} (J);
			\draw[->>,bend right=10] (Q) to node[fill=white, scale=0.5] {$\begin{bmatrix} 1 & 0 & 0 \end{bmatrix}$} (K);
			\draw[->,bend right=10] (I) to node[fill=white,scale=0.65] {$1$} (C);
			\draw[->,bend right=20] (J) to node[fill=white,scale=0.65] {$1$} (C);
			\draw[->,bend right=10] (K) to node[fill=white,scale=0.65] {$1$} (C);
			\draw[->,bend right=20] (C) to node[fill=white,scale=0.65] {$1$} (e);

			\draw[->,bend right=10,\inductioncolor] (I) to node[fill=white,scale=0.5] {$\begin{bmatrix} 2 \\ 1\\ 0\end{bmatrix}$} (Q);
			\draw[->,bend right=10,\inductioncolor] (J) to node[right, scale=0.5] {$\begin{bmatrix} 2 \\ 1\\ 1\end{bmatrix}$} (Q);
			\draw[->,bend right=10,\inductioncolor] (K) to node[above right,scale=0.5] {$\begin{bmatrix} 2 \\ 0\\ 1\end{bmatrix}$} (Q);
			\draw[->,bend right=10,\inductioncolor] (C) to node[fill=white,scale=0.65] {$2$} (I);
			\draw[->,bend right=20,\inductioncolor] (C) to node[fill=white,scale=0.65] {$2$} (J);
			\draw[->,bend right=10,\inductioncolor] (C) to node[fill=white,scale=0.65] {$2$} (K);
			\draw[->,bend right=20,\inductioncolor] (e) to node[fill=white,scale=0.65] {$2$} (C);
		\end{tikzpicture}
		\caption{The Mackey functor $\QN = N_{C_2}^{Q_8}\ulZ$.}
		\label{Q8norm Mackey functor}
	\vspace*{-1cm}
\end{wrapfigure}

In the case of $N_{C_2}^{Q_8}\MUR$, the Mackey functor $\upi_0 \MUR$ is the constant $C_2$-Mackey functor $\ulZ$, and the zero slice is the Eilenberg--Mac~Lane spectrum of the $Q_8$-Mackey functor $N_{C_2}^{Q_8}\ulZ$. We compute this Mackey functor, thereby computing the zero slice of $N_{C_2}^{Q_8} \MUR$.

\begin{maintheorem}[{\zcref{thm:NZ}}]
\label{zero slice of MUR}
The zero slice of $N_{C_2}^{Q_8} MU\R$ is $\HQ \QN$, where $\QN = N_{C_2}^{Q_8}\ulZ$ is the Mackey functor in \zcref{Q8norm Mackey functor}.
\end{maintheorem}

Computing the norm $N_{C_2}^{Q_8} \ulZ$ is \emph{a priori} hindered by the fact that $N_H^G \colon \Mack^{H} \to \Mack^G$ is neither a left adjoint nor an additive functor (see \cref{Mackey norm is not a left adjoint}). We instead compute this norm by recognizing that $\ulZ$ is a Tambara functor, where the norm $n_H^G \colon \Tamb^H \to \Tamb^G$ \emph{is} a left adjoint. Writing $\ulZ$ as a pushout in the category of $C_2$-Tambara functors, we interpret $n_H^G\ulZ$ as the quotient of the $Q_8$-Burnside Tambara functor by a certain Tambara ideal, which we can compute. Finally, a theorem of Hoyer shows that the Mackey functor underlying $n_H^G\ulZ$ is $N_H^G\ulZ$. 

\begin{remark}
	Although the norm functor $\widehat n_H^G \colon \Green^H \to \Green^G$ is also a left adjoint, it does not coincide with $N_H^G$ on the underlying Mackey functors. Indeed, Green functors do not satisfy the hypotheses of \cite{BH2018}*{Theorem~6.15}, and the Mackey functor underlying $\widehat n_e^{C_2}(\Z/2)$ is not $N_{e}^{C_2}(\Z/2)$. In particular, $\widehat n_e^{C_2}(\Z/2)$ is isomorphic $\Z \oplus \Z/2$ at the $C_2/C_2$ level, while $N_e^{C_2}(\Z/2)$ is isomorphic to $\Z/4$. 
	Therefore, the strategy for computing norms of Mackey functors outlined in the previous paragraph requires Tambara functors; it will not work with Green functors. 
\end{remark}

With the understanding of the zero slice of $N_{C_2}^{Q_8}\MUR$ provided by \cref{zero slice of MUR}, we compute a bigraded sector of its $\RO(Q_8)$-graded homotopy Mackey functors.
As in \cite{GKM}, we denote by $\blacklozenge$ a grading over the subgroup $\Z\{1,\orho\}$ of $RO(Q_8)$.

\begin{maintheorem}
\label{thm:htpHN}
The $\blacklozenge$-graded homotopy Mackey functors of $\HQ \QN$ are as depicted in \zcref{fig:posconeN,fig:negconeN}. In degrees $n + k \orho$ with $n \neq 0, 1$, the homotopy Mackey functors of $\HQ \QN$ are isomorphic to those of $\HQ \ulZ$.
\end{maintheorem}

\zcref{thm:htpHN} gives input to the slice spectral sequence for $N_{C_2}^{Q_8} MU\R$. Computing this spectral sequence would also require knowledge of the higher slices. The Slice Theorem of Hill-Hopkins-Ravenel describes the slices of $N_{C_2}^{C_{2^n}} MU\R$ \cite{HHR}*{Theorem~6.1}. Their results says that the slices are all a wedge sum of terms  of the form $ {C_{2^n}}_+\wedge_{L} \Sigma^{i\rho_L} H \ulZ$, for $L$ a subgroup of $C_{2^n}$. We conjecture a similar formula holds for the quaternionic norm of $\MUR$.

\begin{conjecture}
\label{SliceConj}
The slices of $N_{C_2}^{Q_8} MU\R$ are of the form
\[
	P_n^n N_{C_2}^{Q_8} MU\R \simeq \bigvee {Q_8}_+\wedge_{L} \Sigma^{i\rho_L} H N_{C_2}^{L} \ulZ,
\]
for subgroups $L$ of $Q_8$.
\end{conjecture}

As evidence for this conjecture, we note that our conjectural slices have the correct form upon restriction to any $C_4$-subgroup, by the Slice Theorem of \cite{HHR}. Indeed, for any order 4 subgroup $C \cong C_4$ of $Q_8$, then the restriction of $N_{C_2}^{Q_8}MU\R$ to $C$ is 
\[
	\downarrow^{Q_8}_{C} N_{C_2}^{Q_8} MU\R \simeq \left( N_{C_2}^{C_4} MU\R \right)^{\wedge 2} \!\! \simeq \,\, \downarrow^{C_8}_{C_4} N_{C_2}^{C_8} MU\R
\]

\bigskip

While none of {the} calculations in this paper depend on any computer algebra system, we have written some code for algorithmically computing restrictions, norms, and transfers in Burnside Tambara functors, which may be of independent interest \cite{GKMcode}. This code supplements the material in \cref{quaternionic norm of constant Z} below and verifies the calculations therein. 

\subsection*{Conventions}

\begin{itemize}
	\item We write $Q = Q_8$ for the Quaternion group, and $I,J,K$ for its three subgroups of order $4$, and $C_2$ for the subgroup of order 2. 
	\item We write $\phi_H^*$ for the inflation from $Q_8/H$-Mackey functors to $Q_8$-Mackey functors. We write $\phi_{IJK}^*$ as shorthand for the sum $\phi_I^* \oplus \phi_J^* \oplus \phi_K^*$ of inflations from $C_2$-Mackey functors.
	\item We use the notation $\HQ$ for a $Q_8$-equivariant Eilenberg-Mac~Lane spectrum.
\end{itemize}

\subsection*{Acknowledgements}

We thank Mike Hill, XiaoLin Danny Shi, and Zhouli Xu for helpful discussions
and David Chan, Ben Spitz, and Chase Vogeli for help with code.
We acknowledge support from the National Science Foundation via DMS-2403798 and DMS-2135884
and from the Simons Foundation via MPS-TSM-00007067.

\section{The Quaternionic Norm of \texorpdfstring{$\underline{\mathbb{Z}}$}{Constant Z}}
\label{quaternionic norm of constant Z}

As in \cite{GKM}, we compute the Mackey functor $N_{C_2}^{Q_8} \underline{\Z}$ as a quotient of the Burnside Tambara functor by a certain Tambara ideal. 

Recall that a $G$-Mackey functor $\underline{M}$ is a product-preserving functor $\underline{M} \colon \cA^G \to \Ab$, where $\cA^G$ is the Burnside category, and a morphism of $G$-Mackey functors is a natural transformation. We write $\Mack^G$ for the category of $G$-Mackey functors {and $\Fin^G$ for the category of finite $G$-sets and $G$-equivariant functions.}

\begin{definition}[{\cites{Ho,HM}}]
	Let $H$ be a subgroup of a finite group $G$. The norm functor $N_H^G \colon \Mack^H \to \Mack^G$ is  left Kan extension along  coinduction \mbox{$\Fin^H(G,-) \colon \cA^H \to \cA^G$.}
\end{definition}

\begin{remark}
\label{Mackey norm is not a left adjoint}
The norm $N_H^G$ is not a left adjoint. 
Indeed, 
the functor $\Fin^H(G,-)$, considered as a functor from finite $H$-sets to finite $G$-sets, preserves pullbacks but {\it not} disjoint unions. It follows that coinduction induces a {\it non-additive} functor $\cA^H \to \cA^G$. 
Therefore, it cannot induce a functor from $\Mack^G$ to $\Mack^H$ via pullback, and thus the Kan extension $N_H^G$ has no right adjoint (\cite{Ho}*{discussion following Definition 2.3.2}).
\end{remark}

Computing norms of Mackey functors can be difficult because the norm is neither an additive functor nor a left adjoint.
However, we can compute norms efficiently using Tambara functors. Recall that a $G$-Tambara functor is a $G$-commutative monoid in the category of $G$-Mackey functors (\cite{Ho}*{Theorem~2.7.4}, \cite{HM}*{Theorem~1.1}). 
Let $\Tamb^G$ be the category of $G$-Tambara functors. There are forgetful functors $U^G \colon \Tamb^G \to \Mack^G$ for each finite group $G$.

\begin{proposition}[{\cite{Ho}*{Theorem~2.3.3}\cite{BH2018}*{Theorem~6.15}}]
\label{TambMackNorms}
	Let $G$ be a finite group with subgroup $H$. The restriction functor from $G$-Tambara functors to $H$-Tambara functors has a left adjoint $n_H^G \colon \Tamb^H \to \Tamb^G$ such that the following diagram commutes: 
	\[
	\begin{tikzcd}
		\Tamb^H \ar[r, "n_H^G"] \ar[d, "U^H"] & \Tamb^G \ar[d, "U^G"] \\
		\Mack^H \ar[r, "N_H^G"] & \Mack^G
	\end{tikzcd}
	\]
\end{proposition}

The functor $n_H^G \colon \Tamb^H \to \Tamb^G$ is called the \emph{norm} of Tambara functors, and 
\zcref{TambMackNorms} states that
 it agrees with the norm $N_H^G$ of underlying Mackey functors. Thus, we may use the Tambara norm $n_H^G$ to compute the Mackey norm $N_H^G$. The strategy to compute $N_H^G\underline{M}$ is as follows:  
\begin{enumerate}
	\item write $\underline{M}$ as $\underline{M} = U^H \underline{T}$, where $\underline{T}$ is an $H$-Tambara functor. (This is not always possible.)
	\item Write $\underline{T}$ as a pushout in the category of $H$-Tambara functors, where the three other corners of the pushout diagram are $H$-Tambara functors whose norms are easy to compute. Typical choices for the other three corners are Burnside Tambara functors $\uA$ or free Tambara functors $\uA[x_{H/K}]$. 
	\item Apply the Tambara norm $n_H^G$ to the pushout diagram. Since $n_H^G$ is a left adjoint, it preserves pushouts. Compute $n_H^G \underline{T}$ as a pushout in the category of $G$-Tambara functors.
	\item The underlying $G$-Mackey functor of $n_H^G\underline{T}$ is  $N_H^G \underline{M}$ by \zcref{TambMackNorms}. 
\end{enumerate}

\begin{example}
	The following diagram expresses $\Z/2$ as a pushout in the category of $e$-Tambara functors (equivalently, the category of commutative rings): 
	\[
		\begin{tikzcd}
			\Z[x] \ar[dr, "\ulcorner" description, very near end, phantom]\ar[r, "x \mapsto 2"] \ar[d, "x \mapsto 0"'] & \Z \ar[d] \\
			\Z \ar[r] & \Z/2.
		\end{tikzcd}
	\]
	Applying the Tambara norm $n_e^{C_2} \colon \CRing \simeq \Tamb^e \to \Tamb^{C_2}$ yields
	\[
		\begin{tikzcd}
			\uA[x_{C_2/e}] \ar[dr, "\ulcorner" description, very near end, phantom]\ar[r, "x \mapsto 2"] \ar[d, "x \mapsto 0"'] & \uA \ar[d] \\
			\uA \ar[r] & n_e^{C_2}(\Z/2),
		\end{tikzcd}
	\]
	where $\uA[x_{C_2/e}]$ is the free $C_2$-Tambara functor on an underlying generator \cite{BHRightAdj}*{Lemma 3.7}.
	This Tambara functor is created by starting with the Burnside Tambara functor $\uA$, and adding a variable $x$ at level $C_2/e$ and all of its images under the Weyl action and the structure maps of the Tambara functor. Hence, the image of $x\mapsto 2$ inside the Burnside Tambara functor $\uA$ is $2 \in \uA(C_2/e)$ and all of its images under the structure maps of the Tambara functor. This is the \emph{Tambara ideal of $\uA$ generated by $2 \in \uA(C_2/e)$}. Since we are setting $x$ to zero in the other leg of the pushout diagram, we can understand this norm $n_e^{C_2}(\Z/2)$ as a quotient 
	\[
		n_e^{C_2}(\Z/2) \cong \uA/\underline{\mathfrak{a}},
	\]
	where $\underline{\mathfrak{a}}$ is this Tambara ideal described above. Concretely, we have: 
	\[
		\begin{tikzcd}[row sep=normal]
			(2 + t, 2t) 
				\ar[dd] 
				\ar[r, hook]
				& 
			\Z[t]/(t^2 - 2t) 
				\ar[dd]
				\ar[r, two heads]
				& 
			\Z/4
				\ar[dd]
				\\ \\
			(2) 
				\ar[uu, bend left=30, \normcolor]
				\ar[uu, bend right=30, \inductioncolor]
				\ar[r, hook]
				& 
			\Z 	
				\ar[uu, bend left=30, \normcolor]
				\ar[uu, bend right=30, \inductioncolor]
				\ar[r, two heads]
				& 
			\Z/2 
				\ar[uu, bend left=30, \normcolor]
				\ar[uu, bend right=30, \inductioncolor]
				\\
			\underline{\mathfrak{a}} 
				\ar[r, hook]
				& 
			\uA 
				\ar[r, two heads]
				& 
			\uA/\underline{\mathfrak{a}}
		\end{tikzcd}
	\]	
\end{example}

To apply this strategy in the case of $N_{C_2}^{Q_8} \underline{\Z}$, note that we may write
the $C_2$-Tambara functor
 $\underline{\Z}$ as the quotient of $\uA$ by the Tambara ideal generated by $t -2 \in \uA(C_2/C_2)$, where $t = [\sfrac{C_2}{e}]$ is the class of a free orbit in the Burnside ring $A(C_2) = \uA(C_2/C_2)$.  
In the category of $C_2$-Tambara functors, this is a pushout: 
\[
	\begin{tikzcd}
			\uA[x_{C_2/C_2}] \ar[dr, "\ulcorner" description, very near end, phantom]\ar[r, "x \mapsto t-2"] \ar[d, "x \mapsto 0"'] & \uA \ar[d] \\
			\uA \ar[r] & \ulZ,
		\end{tikzcd}
\]
Applying the norm $n_{C_2}^{Q_8}$ yields another pushout diagram 
\[
	\begin{tikzcd}
			\uA[x_{Q_8/C_2}] \ar[dr, "\ulcorner" description, very near end, phantom]\ar[r, "x \mapsto t-2"] \ar[d, "x \mapsto 0"'] & \uA \ar[d] \\
			\uA \ar[r] & n_{C_2}^{Q_8}\ulZ,
		\end{tikzcd}
\]
where we have used \cite{HMQ}*{Proposition 4.0.2} to compute $n_{C_2}^{Q_8} \uA[x_{C_2/C_2}]$. This diagram expresses $n_{C_2}^{Q_8} \ulZ$ as the quotient of the $Q_8$-Burnside Tambara functor $\uA$ by the Tambara ideal
generated by $t-2 \in \uA(Q_8/C_2)$. Thus, to compute $N_{C_2}^{Q_8} \ulZ$, we must first compute this ideal.

\begin{definition}
	Let $\ufa$ be the Tambara ideal of the $Q_8$-Burnside Tambara functor $\uA$ generated by $t-2 = [\sfrac{C_2}{e}] - 2 \in \uA(Q_8/C_2)$. 
\end{definition}

Note that $\ufa$ is a principal ideal. 

\begin{remark}
\label{contained in augmentation ideal}
Because $t-2$ is sent to zero by the augmentation homomorphism $\uA \to \ulZ$, we know that the principal ideal $\ufa$ generated by $t-2$ is contained in the augmentation ideal $\augmentationideal = \ker(\uA \to \ulZ)$. We will show that these ideals in fact agree except at the fixed level $\uA(Q_8/Q_8)$. 
\end{remark}

We need one preliminary computation.

\begin{lemma}
\label{lemma:norm c2 c4 t-2}
	In the $C_4$-Burnside Tambara functor, the norm of $t-2$ is given by
	\[
		\nm_{C_2}^{C_4}(t - 2) = -[\sfrac{C_4}{e}] + 3 [\sfrac{C_4}{C_2}] - 2. 
	\]
\end{lemma}

\begin{proof}
	We use the formula for the norm of a sum from \cite{HM}*{Corollary~2.6} specialized to the case of trivial Weyl actions: 
	\[
		\nm_{C_2}^{C_4}(a + b) = \nm_{C_2}^{C_4}(a) + \nm_{C_2}^{C_4}(b) + \tr_{C_2}^{C_4}(ab). 
	\]
	Using this formula, we can compute the norm of $2$: 
	\begin{align*}
		\nm_{C_2}^{C_4}(2) &= \nm_{C_2}^{C_4}(1) + \nm_{C_2}^{C_4}(1) + \tr_{C_2}^{C_4}(1) \\
		&= 1 + 1 + \tr_{C_2}^{C_4}(1) \\
		&= 2 + [\sfrac{C_4}{C_2}].
	\end{align*}
	Similarly, we use this formula to compute the norm of $-1$: 
	\begin{align*}
		\nm_{C_2}^{C_4}(0) &= \nm_{C_2}^{C_4}(1 + -1)\\
		& = \nm_{C_2}^{C_4}(1) + \nm_{C_2}^{C_4}(-1) + \tr_{C_2}^{C_4}(-1) \\
		& = 1 + \nm_{C_2}^{C_4}(-1) - [\sfrac{C_4}{C_2}].
	\end{align*}	
	But $\nm_{C_2}^{C_4}(0) = 0$, so 
	\[
		\nm_{C_2}^{C_4}(-1) = [\sfrac{C_4}{C_2}] - 1. 
	\]
	
	To compute $\nm_{C_2}^{C_4}(t) = \nm_{C_2}^{C_4}([\sfrac{C_2}{e}])$, however, we resort to the definition. This norm is the class of the coinduction $[\Fin^{C_2}(C_4,\sfrac{C_2}{e})]$, which is a single free $C_4$-orbit:
	\[
		\nm_{C_2}^{C_4}(t) = [\sfrac{C_4}{e}]. 
	\]

	Finally, we can use the formula for the norm of a sum to calculate 
	\begin{align*}
		\nm_{C_2}^{C_4}(t-2) &= \nm_{C_2}^{C_4}(t) + \nm_{C_2}^{C_4}(-2) + \tr_{C_2}^{C_4}(-2t) \\
		&= [\sfrac{C_4}{e}] + \nm_{C_2}^{C_4}(2)\nm_{C_2}^{C_4}(-1) -2 [\sfrac{C_4}{e}] \\
		&= -[\sfrac{C_4}{e}] + \left( 2 + [\sfrac{C_4}{C_2}]\right)\left([\sfrac{C_4}{C_2}] - 1\right)\\
		&= -[\sfrac{C_4}{C_2}] + 3 [\sfrac{C_4}{C_2}] - 2 \qedhere
	\end{align*}	
\end{proof}

\begin{proposition}
\label{QuotientC4Value}
	Let $C_4$ be any order 4 subgroup of $Q_8$. 
	Then $(\uA/\ufa)(Q_8/C_4)$ is isomorphic to $\Z$ as a commutative ring.
\end{proposition}

\begin{proof}
	By \zcref{contained in augmentation ideal}, we know that $\ufa(Q_8/C_4)$ is contained in $\augmentationideal(Q_8/C_4)$. We will show the opposite inclusion, whence $(\uA/\ufa)(Q_8/C_4)$ is isomorphic to $(\uA/\augmentationideal)(Q_8/C_4) \cong \Z$. 
	
	The augmentation ideal $\augmentationideal$ is the ideal of virtual $G$-sets of virtual cardinality zero. At level $Q_8/C_4$, the generators of $\augmentationideal(Q_8/C_4)$ are $[\sfrac{C_4}{C_2}] - 2$ and $[\sfrac{C_4}{e}] - 4$. We will show that these are in the ideal $\ufa(Q_8/C_4)$. 
	
	The ideal $\ufa(Q_8/C_4)$ contains the elements  
	\begin{align*}
		\tr_{C_2}^{C_4}(t-2) &= \hphantom{-}[\sfrac{C_4}{e}] - 2[\sfrac{C_4}{C_2}]\\
	  \text{ and } \qquad	\nm_{C_2}^{C_4}(t-2) &= -[\sfrac{C_4}{e}] + 3 [\sfrac{C_4}{C_2}] - 2.
	\end{align*}
	The generators of $\augmentationideal(Q_8/C_4)$ can be written as linear combinations of the elements above:
	\begin{align*}
		[\sfrac{C_4}{C_2}] - 2 &= \hphantom{3}\tr_{C_2}^{C_4}(t-2) + \hphantom{2}\nm_{C_2}^{C_4}(t-2)  \\
		[\sfrac{C_4}{e}]\ \  - 4 &= 3 \tr_{C_2}^{C_4}(t-2) + 2 \nm_{C_2}^{C_4}(t-2).
	\end{align*}
	Hence, the generators of $\augmentationideal(Q_8/C_4)$ are in $\ufa(Q_8/C_4)$. 
\end{proof}

\begin{remark}
\label{rem:Nakaoka}
	The proof of \cref{QuotientC4Value} shows that $\tr_{C_2}^{C_4}(t-2)$ and $\nm_{C_2}^{C_4}(t-2)$ additively generate the ideal $\ufa(Q_8/C_4)$.
	This also follows from \cite{ideals}*{Proposition 3.4}. 
\end{remark}

We now turn to the $Q_8$-level.

\begin{lemma}
\label{lemma:generators at Q/Q}
	The ideal $\ufa(Q_8/Q_8)$ of $\uA(Q_8/Q_8)$ is generated by $\nm_{C_2}^{Q_8}(t-2)$, $\tr_{C_2}^{Q_8}(t-2)$, and $\tr_H^{Q_8} \nm_{C_2}^H(t-2)$ for $H \in \{I,J,K\}$. 
\end{lemma}

\begin{proof}
	Let $\mathfrak{b}$ be the ideal of the ring $\uA(Q_8/Q_8)$ generated by these elements. This ideal $\mathfrak{b}$ is contained in $\ufa(Q_8/Q_8)$; we must show the other inclusion.
	
	By \cite{ideals}*{Proposition 3.4}, any element of $\ufa(Q_8/Q_8)$ is of the form 
	\[
		\tr(y \cdot \nm(\res(t-2))), 
	\] 
	by which we mean any element of $\ufa(Q_8/Q_8)$ is a transfer of a norm of a restriction (in that order) of $t-2 \in \uA(Q_8/C_2)$, possibly multiplying the norm by an element $y$ of $\uA$ at the same level before transferring. Note also that the transfers, norms, and restrictions in the formula above are between possibly non-transitive $Q_8$-sets. Consequently, if the map along which we transfer contains a fold $\nabla \colon X \amalg X \to X,$ then this transfer involves a sum. Likewise, the norm may involve a product and the restriction may duplicate terms. 
	
	To understand elements of $\ufa(Q_8/Q_8)$, we consider all possible paths along which we can restrict then transfer then norm $t-2$ and land in $\uA(Q_8/Q_8)$. Since $\res^{C_2}_e(t-2) = 0$, 
the only relevant restrictions are duplications.
Without any duplication, this procedure produces a term of the form $\tr_H^{Q_8}(y \cdot \nm_{C_2}^H(t-2))$ for $C_2  \leq H \leq Q_8$ and $y\in A(H)$. 
When $t-2$ is duplicated in restriction, however, duplicated terms are recombined either at the norm stage by multiplication or at the transfer stage by addition. 
If duplicates are recombined by multiplication in the norm, extra factors of $\nm_{C_2}^H(t-2)$ may be absorbed into the $y$. 
If duplicates are recombined in the transfer, we add them together. Therefore,  a general element in $\ufa(Q_8/Q_8)$ may be written as a sum of terms of the form 
	\begin{equation}
	\label{general element}
		\tr_H^{Q_8}(y \cdot \nm_{C_2}^H (t-2))
	\end{equation}
	for some $H$ such that $C_2 \leq H \leq Q_8$. There are five possibilities for $H$: any of the nontrivial subgroups of $Q_8$. 
	\begin{itemize}
		\item If $H = C_2$, then the norm is an identity and \cref{general element} becomes $\tr_{C_2}^{Q_8}(y \cdot (t-2))$ for some $y \in \uA(Q_8/C_2)$. But for any $y = a + bt \in \uA(Q_8/C_2)$ with $a,b \in \Z$, 
		\[
			(a + bt)(t-2) = a(t-2) + b(t^2 - 2t) = a (t-2),
		\]
		so $\tr_{C_2}^{Q_8}(y(t-2)) = a \tr_{C_2}^{Q_8}(t-2)$ for some $a \in \Z$, which is contained in $\mathfrak{b}$.
				
		\item If $H = Q_8$, then the transfer is an identity and \cref{general element} becomes $y \cdot \nm_{C_2}^{Q_8}(t-2)$, which is in the ideal $\mathfrak{b}$. 
		
		\item If $H \cong C_4$ is an order 4 subgroup, then \cref{general element} is $\tr_{C_4}^{Q_8}(y \cdot \nm_{C_2}^{C_4}(t-2))$ for some $y \in \uA(Q_8/{C_4})$. Recall {from the proof of \cref{QuotientC4Value}} that 
			\[
				\nm_{C_2}^{C_4}(t-2) = -[\sfrac{C_4}{e}] + 3 [\sfrac{C_4}{C_2}] - 2,
			\] 
			and let 
			\[
				y = a[\sfrac{C_4}{e}] + b [\sfrac{C_4}{C_2}] + c \in A({C_4}) \cong \uA(Q_8/{C_4})
			\] 
			with $a, b, c \in \Z$. Then we can calculate:
		\begin{align*}
			\hspace*{1cm}\tr_{C_4}^{Q_8}(y \cdot \nm_{C_2}^{C_4}(t-2)) 
				&= \tr_{C_4}^{Q_8}\bigg((-2b-c)[\sfrac{C_4}{e}] + (4b + 3c)[\sfrac{C_4}{C_2}] - 2c \bigg)\\
				&= \tr_{C_4}^{Q_8}\bigg(-2b\bigg([\sfrac{C_4}{e}]-2[\sfrac{C_4}{C_2}]\bigg) + c \bigg(-[\sfrac{C_4}{e}] + 3[\sfrac{C_4}{C_2}] - 2 \bigg) \bigg)\\
				&= \tr_{C_4}^{Q_8}\bigg( 
					-2b \tr_{C_2}^{C_4}(t-2) + c \nm_{C_2}^{C_4}(t-2)
				\bigg)\\
				&= -2b \tr_{C_2}^{Q_8}(t-2) + c \tr_{C_4}^{Q_8} \nm_{C_2}^{C_4}(t-2). 
		\end{align*}
		In particular, this is a $\Z$-linear combination of elements in $\mathfrak{b}$. 
	\end{itemize}
	Therefore, any element of $\ufa(Q_8/Q_8)$ is a $\Z$-linear combination of elements of $\mathfrak{b}$, so $\ufa(Q_8/Q_8)$ is contained in $\mathfrak{b}$. 
\end{proof}

We record here
 that the generators of the ideal $\ufa(Q_8/Q_8) $ are
\[
\begin{split}
\nm_{C_2}^{Q_8}(t-2) &= -2 [\sfrac{Q_8}{e}] + 3 [\sfrac{Q_8}{I}] + 3 [\sfrac{Q_8}{J}] + 3 [\sfrac{Q_8}{K}] -2 ,
\\
\tr_{C_2}^{Q_8}(t-2) &= [\sfrac{Q_8}{e}] - 2 [\sfrac{Q_8}{C_2}],
\\
\tr_I^{Q_8} \nm_{C_2}^I(t-2) &= -[\sfrac{Q_8}{e}] + 3 [\sfrac{Q_8}{C_2}] -2 [\sfrac{Q_8}{I}],
\\
\tr_J^{Q_8} \nm_{C_2}^J(t-2) &= -[\sfrac{Q_8}{e}] + 3 [\sfrac{Q_8}{C_2}] -2 [\sfrac{Q_8}{J}],
\\
\text{and} \qquad \quad \tr_K^{Q_8} \nm_{C_2}^K(t-2) &= -[\sfrac{Q_8}{e}] + 3 [\sfrac{Q_8}{C_2}] -2 [\sfrac{Q_8}{K}].
\end{split}
\]
The last three of these are obtained by transferring the result of \cref{lemma:norm c2 c4 t-2} to $Q_8$, and the first comes from \cref{lemma:norm c2 q8 t-2} below. 

\begin{lemma}
\label{lemma:norm c2 q8 t-2}
	In the $Q_8$-Burnside Tambara functor, the norm of $t-2$ is given by
	\[
		\nm_{C_2}^{Q_8}(t - 2) = 
		-2 [\sfrac{Q_8}{e}] + 3 [\sfrac{Q_8}{I}] + 3 [\sfrac{Q_8}{J}] + 3 [\sfrac{Q_8}{K}] -2 . 
	\]
\end{lemma}

\begin{proof}
	Because norms are functorial, we have  
	\[
		\nm_{C_2}^{Q_8}(t-2) 
			= \nm_{I}^{Q_8}\nm_{C_2}^I(t-2)
			= \nm_{I}^{Q_8}(-[\sfrac{I}{e}] + 3 [\sfrac{I}{C_2}] - 2),
	\]
	where the identification of $\nm_{C_2}^I(t-2)$ comes from \cref{lemma:norm c2 c4 t-2}. We can again compute this using the formula for the norm of a sum from \cite{HM}*{Theorem 2.5}, which is stated for finite abelian groups, but holds more generally for finite Dedekind groups like $Q_8$.  
	In this Tambara functor, the Hill--Mazur formula takes the form: 
	\begin{equation}
	\label{norm of sum I to Q}
		\nm_{I}^{Q_8}(a + b) =  \nm_{I}^{Q_8}(a) + \nm_{I}^{Q_8}(b) + \tr_{I}^{Q_8}(ab). 
	\end{equation}
	To apply this formula, we must know 
	\begin{align*}
		\nm_{I}^{Q_8}([\sfrac{I}{e}]) &= [\Fin^{I}(Q_8,I/e)] = 2[\sfrac{Q_8}{e}] \\
		\nm_I^{Q_8}([\sfrac{I}{C_2}]) &= [\Fin^I(Q_8,I/C_2)] = [\sfrac{Q_8}{J}] + [\sfrac{Q_8}{K}],
	\end{align*}
	which can be found by counting orbits. 
\end{proof}

We now identify the quotient of $A(Q_8)$ by this ideal.

\begin{proposition}
	The quotient of $\uA(Q_8/Q_8)$ by $\ufa(Q_8/Q_8)$ is 
	\[
		(\uA / \ufa)(Q_8/Q_8) \cong \Z[\bar u_I,\bar u_K] / (\bar u_I^2, \bar u_K^2, \bar u_I \bar u_K, 2 \bar u_I, 2 \bar u_K),
	\]
	where $\bar u_I$ is the image of $[\sfrac{Q_8}{I}] - 2$ and $\bar u_K$ is the image of $[\sfrac{Q_8}{K}] - 2$.
\end{proposition}
	
\begin{proof}
Consider the ring homomorphism $q\colon A(Q_8) \to \Z[\bar u_I,\bar u_K] / (\bar u_I^2, \bar u_K^2, \bar u_I \bar u_K, 2 \bar u_I, 2 \bar u_K)$ defined by
\[
	q( [\sfrac{Q_8}{e}] ) = 8, \qquad q( [\sfrac{Q_8}{C_2}] ) = 4
\]
\[
	q( [\sfrac{Q_8}{I}]) = \bar u_I + 2, \qquad q( [\sfrac{Q_8}{J}] ) = \bar u_I + \bar u_K + 2, \qquad q( [\sfrac{Q_8}{K}]) = \bar u_K + 2.
\]

The homomorphism $q$ annihilates the generators of the ideal $\ufa ( Q_8/Q_8 )$ listed above. On the other hand, $q$ is a group homomorphism between finitely generated abelian groups, of the form $q\colon \Z^6 \to \Z \oplus \Z/2^2$.
Thus the kernel can be computed as the intersection of the kernels of the corresponding homomorphisms from $\Z^6$ to $\Z$ and to $\Z/2$. A set of generators for this kernel is
\[
	r_1 = [\sfrac{Q_8}{e}] - 2 [\sfrac{Q_8}{C_2}], \qquad r_2 = [\sfrac{Q_8}{C_2}] - 2 [\sfrac{Q_8}{I}], 
\]
\[
	r_3 = 2 [\sfrac{Q_8}{I}] - 4 , \qquad r_4 = 2 [\sfrac{Q_8}{K}] - 4 , \qquad r_5 = [\sfrac{Q_8}{I}] + [\sfrac{Q_8}{J}] + [\sfrac{Q_8}{K}] - 6.
\]
It therefore suffices to express the generators $r_1, \dots, r_5$ as linear combinations of the generators of $\ufa( Q_8/Q_8 )$. Such expressions are given by
\[
	r_1 = \tr_{C_2}^{Q_8} (t-2), \qquad r_2 = \tr_I^{Q_8} \nm_{C_2}^I (t-2) + \tr_{C_2}^{Q_8} (t-2),
\]
\[
	r_3 = 2 \nm_{C_2}^{Q_8} (t-2) + 12 \tr_{C_2}^{Q_8} (t-2) + 2 \tr_{I}^{Q_8} \! \nm_{C_2}^I (t-2) + 3 \tr_{J}^{Q_8} \! \nm_{C_2}^J (t-2) + 3 \tr_{K}^{Q_8} \! \nm_{C_2}^K (t-2) ,
\]
\[
	r_4 = 2 \nm_{C_2}^{Q_8} (t-2) + 12 \tr_{C_2}^{Q_8} (t-2) + 3 \tr_{I}^{Q_8} \! \nm_{C_2}^I (t-2) + 3 \tr_{J}^{Q_8} \! \nm_{C_2}^J (t-2) + 2 \tr_{K}^{Q_8} \! \nm_{C_2}^K (t-2) ,
\]
\[
	r_5 = 3 \nm_{C_2}^{Q_8} (t-2) + 18 \tr_{C_2}^{Q_8} (t-2) + 4 \tr_{I}^{Q_8} \! \nm_{C_2}^I (t-2) + 4 \tr_{J}^{Q_8} \! \nm_{C_2}^J (t-2) + 4 \tr_{K}^{Q_8} \! \nm_{C_2}^K (t-2) .
\]
It follows that the kernel of $q$ agrees with $\ufa (Q_8/Q_8)$.
\end{proof}

\begin{theorem}
\label{thm:NZ}
	The normed Tambara functor $n_{C_2}^{Q_8}\underline{\Z}$ and its underlying Mackey functor $N_{C_2}^{Q_8}\underline{\Z}$ are given by the following, where $\bar u_I$ is the image of $[\sfrac{Q_8}{I}] - 2$ and $\bar u_K$ is the image of $[\sfrac{Q_8}{K}] - 2$. 
	\bigskip
	\begin{center}
		\begin{tikzpicture}[yscale=1.25]
			\node(Q) at (0,6) {$\Z[\bar u_I,\bar u_K]/(\bar u_I^2,\bar u_K^2,\bar u_I \bar u_K,2\bar u_I,2\bar u_K)$};
			\node(I) at (-3,3.5) {$\Z$};
			\node(J) at (0,3.5) {$\Z$};
			\node(K) at (3,3.5) {$\Z$};
			\node(C) at (0,2) {$\Z$};
			\node(e) at (0,0) {$\Z$};

			\draw[->] (Q) to node[fill=white,rotate=45,scale=0.65] {$\bar u_I, \bar u_K \mapsto 0$} (I);
			\draw[->] (Q) to node[fill=white,rotate=-90,scale=0.65] {$\bar u_I, \bar u_K \mapsto 0$} (J);
			\draw[->] (Q) to node[fill=white,rotate=-45,scale=0.65] {$\bar u_I, \bar u_K \mapsto 0$} (K);
			\draw[->] (I) to node[fill=white] {\tiny$1$} (C);
			\draw[->] (J) to node[fill=white] {\tiny$1$} (C);
			\draw[->] (K) to node[fill=white] {\tiny$1$} (C);
			\draw[->] (C) to node[fill=white] {\tiny$1$} (e);

			\draw[->,bend right=20,\inductioncolor] (I) to node[fill=white, scale=0.65] {$\bar{u}_I + 2$} (Q);
			\draw[->,bend right=25,\inductioncolor] (J) to node[near start, fill=white,right, rotate=90, scale=0.65] {$\bar{u}_I + \bar{u}_J + 2$} (Q);
			\draw[->,bend right=20,\inductioncolor] (K) to node[fill=white, scale=0.65] {$\bar{u}_K + 2$} (Q);
			\draw[->,bend right=20,\inductioncolor] (C) to node[above right,scale=0.65] {$2$} (I);
			\draw[->,bend right=20,\inductioncolor] (C) to node[right,scale=0.65] {$2$} (J);
			\draw[->,bend right=20,\inductioncolor] (C) to node[below right,scale=0.65] {$2$} (K);
			\draw[->,bend right=20,\inductioncolor] (e) to node[right,scale=0.65] {$2$} (C);
	
			\draw[->,bend left=20,\normcolor] (I) to node[left,rotate=30] {} (Q);
			\draw[->,bend left=20,\normcolor] (J) to node[left,rotate=90] {} (Q);
			\draw[->,bend left=20,\normcolor] (K) to node[left] {} (Q);
			\draw[->,bend left=20,\normcolor] (C) to node[below left,scale=0.65] {$(-)^2$} (I);
			\draw[->,bend left=20,\normcolor] (C) to node[left,scale=0.65] {$(-)^2$} (J);
			\draw[->,bend left=20,\normcolor] (C) to node[above left,scale=0.65] {$(-)^2$} (K);
			\draw[->,bend left=20,\normcolor] (e) to node[left,scale=0.65] {$(-)^2$} (C);
		\end{tikzpicture}
		\hspace*{1cm}
		\begin{tikzpicture}[yscale=1.25]
			\node(Q) at (0,6) {$\Z \oplus \F_2\{\bar u_I,\bar u_K\}$};
			\node(I) at (-3,3.5) {$\Z$};
			\node(J) at (0,3.5) {$\Z$};
			\node(K) at (3,3.5) {$\Z$};
			\node(C) at (0,2) {$\Z$};
			\node(e) at (0,0) {$\Z$};

			\draw[->>,bend right=20] (Q) to node[fill=white,scale=0.5] {$\begin{bmatrix} 1 & 0 & 0 \end{bmatrix}$} (I);
			\draw[->>,bend right=20] (Q) to node[fill=white, scale=0.5] {$\begin{bmatrix} 1 & 0 & 0 \end{bmatrix}$} (J);
			\draw[->>,bend right=20] (Q) to node[fill=white, scale=0.5] {$\begin{bmatrix} 1 & 0 & 0 \end{bmatrix}$} (K);
			\draw[->,bend right=20] (I) to node[fill=white,scale=0.65] {$1$} (C);
			\draw[->,bend right=20] (J) to node[fill=white,scale=0.65] {$1$} (C);
			\draw[->,bend right=20] (K) to node[fill=white,scale=0.65] {$1$} (C);
			\draw[->,bend right=20] (C) to node[fill=white,scale=0.65] {$1$} (e);

			\draw[->,bend right=20,\inductioncolor] (I) to node[fill=white,scale=0.5] {$\begin{bmatrix} 2 \\ 1\\ 0\end{bmatrix}$} (Q);
			\draw[->,bend right=20,\inductioncolor] (J) to node[fill=white, scale=0.5] {$\begin{bmatrix} 2 \\ 1\\ 1\end{bmatrix}$} (Q);
			\draw[->,bend right=20,\inductioncolor] (K) to node[fill=white,scale=0.5] {$\begin{bmatrix} 2 \\ 0\\ 1\end{bmatrix}$} (Q);
			\draw[->,bend right=20,\inductioncolor] (C) to node[above right,scale=0.65] {$2$} (I);
			\draw[->,bend right=20,\inductioncolor] (C) to node[right,scale=0.65] {$2$} (J);
			\draw[->,bend right=20,\inductioncolor] (C) to node[below right,scale=0.65] {$2$} (K);
			\draw[->,bend right=20,\inductioncolor] (e) to node[right,scale=0.65] {$2$} (C);
		\end{tikzpicture}
	\end{center}
	In the Tambara functor $n_{C_2}^{Q_8}\underline{\Z}$, the norms $\nm_{C_4}^{Q_8}$ are determined by $1 \mapsto 1$, $0 \mapsto 0$, and the formula for the norm of a sum \cref{norm of sum I to Q}.
\end{theorem}

Here and in \cref{tab-Q8Mackey}, the Lewis diagrams are displayed with the values at the orbits $Q_8/I$, $Q_8/J$, and $Q_8/K$ from left to right.

\subsection{Comparison to Restrictions and Geometric Fixed Points}
\label{sec:compare}

The restriction of $N_{C_2}^{Q_8} \ulZ$ to any of the $C_4$ subgroups is the constant Mackey functor $\ulZ$. This is expected. First, writing $C_4$ indiscriminately for any choice of $C_4$ subgroup, we can rewrite the norm to $Q_8$ as the two-stage norm
\[
	N_{C_2}^{Q_8} \ulZ = N_{C_4}^{Q_8} N_{C_2}^{C_4} \ulZ.
\]
The $C_4$-norm $N_{C_2}^{C_4} \ulZ$ is isomorphic to $\ulZ$. This can be verified directly, using the pushout technique. Alternatively, it follows from the computation 
\[
	N_{C_2}^{C_4} \ulZ \cong N_{C_2}^{C_4} \upi_0 MU\R \cong \upi_0 N_{C_2}^{C_4} MU\R \cong \ulZ
\]
of \cite{HHR}. But now we have
\[
	\downarrow^{Q_8}_{C_4} \! N_{C_2}^{Q_8} \ulZ \cong \, \downarrow^{Q_8}_{C_4} \! N_{C_4}^{Q_8} \ulZ \cong \ulZ^{\boxtimes 2} \cong \ulZ
\]
by \cite{BGHL}*{Corollary~3.17}.

Similarly, we can consider geometric fixed points, as in \cite{BGHL}*{Section~5.2}. In particular, \cite{BGHL}*{Theorem~5.15} gives
\[
	\Phi^{C_2} N_{C_2}^{Q_8} \ulZ \cong N_e^{C_2\times C_2} \Phi^{C_2} \ulZ \cong N_e^{C_2\times C_2} \F_2.
\]
On the other hand, the geometric fixed points $\Phi^{C_2}$ of a $Q_8$-Mackey functor are the $C_2\times C_2$-Mackey functor obtained by modding out by transfers from the trivial subgroup. This, coupled with \cref{thm:NZ}, recovers the computation of \cite{GKM}*{Figure 2.2}.

\section{\texorpdfstring{$RO(Q_8)$}{RO(Q_8)}-graded Homotopy}
We proceed in the same fashion as in \cite{GKM}, comparing to an Eilenberg-Mac Lane $Q_8$-spectrum whose homotopy is already known. We note that there is a short exact sequence
\begin{equation}
\label{eq:mainses}
\infl{\F_2}^2  \into N_{C_2}^{Q_8} \ulZ \onto \ulZ,
\end{equation}
where $\infl{\F_2}$ is the inflated Mackey functor with $\F_2$ at $Q_8/Q_8$ and zeros elsewhere (see \cref{tab-Q8Mackey}). This short exact sequence yields a cofiber sequence
\begin{equation}
\label{eq:maincofiber}
\HQ \infl{\F_2}^2  \to \HQ N_{C_2}^{Q_8} \ulZ \to  \HQ \ulZ.
\end{equation} 

The efficiency of this approach is captured in the fact that the kernel $\infl{\F_2}^2$ presented in \zcref{eq:mainses} is an inflated Mackey functor. This implies that $\HQ \infl{\F_2}^2$ is invariant under 
suspension by any representation with zero-dimensional fixed points
and gives the following comparison.

\begin{proposition}
\label{prop:augment}
Let $V\in RO(Q_8)$.
The augmentation $N_{C_2}^{Q_8} \ulZ \onto \ulZ$ induces an isomorphism 
\[
	\upi_{V} \HQ N_{C_2}^{Q_8} \ulZ \cong \upi_{V} \HQ \ulZ
\]
when the dimension of the $Q_8$-fixed points of $V$ is not $0$ or $1$.
\end{proposition}

In particular, this says that the augmentation induces an isomorphism in degrees $n+k\orho$ when $n\neq 0,1$.

\begin{proof}
Let us write $V = n + \overline{V}$. Then the claim is that the augmentation induces an isomorphism $\upi_n \Sigma^{-\overline{V}} \HQ N_{C_2}^{Q_8} \ulZ \cong \upi_n \Sigma^{-\overline{V}}   \HQ \ulZ$ for $n$ not equal to 0 or 1. The fiber of the $\overline{V}$-desuspended augmentation is $\Sigma^{-\overline{V}} \HQ \infl{\F_2}^2 \simeq \HQ \infl{\F_2}^2$. The result follows.
\end{proof}

The Mackey functors $\upi_{\blacklozenge} \HQ \ulZ$ were computed in \cite{GS}*{Section 4} and are displayed in \zcref{fig:posconeZ} and \zcref{fig:negconeZ}. 
We collect the homotopy Mackey functors outside of the range of \zcref{prop:augment} in \zcref{prop:manual}.

\begin{proposition}
\label{prop:manual}
The non-zero homotopy Mackey functors $\upi_{k\overline{\rho}}(\HQ N_{C_2}^{Q_8} \ulZ)$ are as follows:  
\[	
\upi_{k\overline{\rho}}(\HQ N_{C_2}^{Q_8} \ulZ) \cong 
\begin{cases}
	\infl{\F_2}^2 & k > 0 \\
	N_{C_2}^{Q_8} \ulZ & k=0 \\
	\infl{\F_2} & k < 0. 
\end{cases}
\]
The only non-zero homotopy Mackey functor among $\upi_{1+k\overline{\rho}}(\HQ N_{C_2}^{Q_8} \ulZ)$ for $k \in \Z$ is 
\[ 
	\upi_{1-\overline{\rho}}(\HQ N_{C_2}^{Q_8} \ulZ) \cong \phi_{IJK}^* \ulf.
\]  
\end{proposition}

We will use the following lemma in the proof of \zcref{prop:manual}, which is essentially \cite{BGHL}*{Proposition~5.6}.

\begin{lemma} 
\label{lem:pi0Sigorho}
For any nontrivial finite group $G$ and $G$-Mackey functor $\ul{M}$, the Mackey functor $\upi_0 \Sigma^{\orho} H\ul{M}$ is the inflated Mackey functor
\[
	\upi_0 \Sigma^{\orho} H\ul{M} \cong \Phi^G \ul{M} = \infl{ \ul{M}(G/G)/\mathrm{tr} },
\]
where $\ul{M}(G/G)/\mathrm{tr}$ denotes the quotient of $\ul{M}(G/G)$ by all transfers from proper subgroups.
\end{lemma}

\begin{proof}[Proof of \zcref{prop:manual}]
We use the cofiber sequence \zcref{eq:maincofiber}. 
In this proof, we will abbreviate the Mackey functor $N_{C_2}^{Q_8} \ulZ$ to $\QN$.

For $k > 0$, it is shown in \cite{GS}*{Proposition~4.12}, and displayed in \zcref{fig:negconeZ}, that $\upi_{n+k\orho} \HQ \ulZ$ is concentrated in degrees $n + k \orho$ with $n \leq -4$. Thus the map $\Sigma^{-k\orho} \HQ \infl{\F_2}^2 \to \Sigma^{-k\orho} \HQN$ induces an isomorphism on $\upi_n$ for $n \geq -3$. As $\Sigma^{-k\orho} \HQ \infl{\F_2}^2 \simeq \HQ \infl{\F_2}^2$ is concentrated in degree 0, the claim about $ \upi_{n+k\orho}\HQN$ follows.

The claim about $k=0$ is true by definition of an Eilenberg-Mac~Lane spectrum. Consider now the case of $k=-1$. Here, we wish to compute $\upi_n \Sigma^{\orho} \HQN$ for $n\in \{0,1\}$. By \zcref{lem:pi0Sigorho}, it follows that $\upi_0 \Sigma^{\orho}\HQN$ is $\infl{\F_2}$. Furthermore, the augmentation $\QN \to \ulZ$ induces an isomorphism modulo transfers, so that $\Sigma^{\orho} \HQN \to \Sigma^{\orho} \HQ \ulZ$ induces an isomorphism on $\upi_0$. The long exact sequence 
\[
\begin{tikzcd}[arrows={thick},]
		&
	\Sigma^{\orho} \HQ \infl{\F_2}^2 \simeq \HQ \infl{\F_2}^2
		&
	\Sigma^{\orho} \HQN
		&
	\Sigma^{\orho} \HQ \ulZ
		\\
	n = 1
		&
	\phantom{0}
		&
	??
		\ar[r] 
		\ar[d, phantom, ""{coordinate, name=A}]
		&
	\ul{mg}
		\ar[dll, rounded corners, 
			to path = { -- ([xshift=5mm]\tikztostart.east)
						|- (A)
						-| ([xshift=-5mm]\tikztotarget.west)
						-- (\tikztotarget)
						}] 
		\\
	n = 0
		&
		\infl{\F_2}^2 
		\ar[r] 
		& 
		\infl{\F_2}
		\ar[r]
		&
	\infl{\F_2} 
\end{tikzcd}
\] 
 shows that $\upi_1 \Sigma^{\orho} \HQN$ is the kernel of a surjection $\ul{mg} \to \infl{\F_2}^2$.
 
 Next, for the case $k=-2$, we may start with the cofiber sequence for the connected cover
 \[
 	P_{\geq 1} \Sigma^{\orho} \HQN \to \Sigma^{\orho} \HQN \to  \HQ \infl{ \F_2 }
 \]
 Suspending by $\orho$ gives 
 \[
 	\Sigma^{\orho} P_{\geq 1} \Sigma^{\orho} \HQN \to \Sigma^{2\orho} \HQN \to  \Sigma^{\orho} \HQ \infl{ \F_2 } \simeq \HQ \infl{ \F_2 }.
 \]
 But suspending the next covering sequence for $P_{\geq 1} \Sigma^{\orho} \HQN $ by $\orho$ gives
 \[
 	\Sigma^{\orho} P_{\geq 2} \Sigma^{\orho} \HQN \to \Sigma^{\orho} P_{\geq 1} \Sigma^{\orho} \HQN \to  \Sigma^{\orho+1} \HQ \phi^*_{IJK} \ulf  \simeq \Sigma^2 \HQ \phi^*_{IJK} \ulF.
 \]
 It follows that $\upi_{1-2\orho} \HQN$ vanishes. 
 
 Further suspending by $\orho$ similarly yields the vanishing of $\upi_{1+k\orho} \HQN$ for larger negative values of $k$.
\end{proof}

\begin{table} 
\caption{Some $Q_{8}$-Mackey functors}
\label{tab-Q8Mackey}
{\renewcommand{\arraystretch}{1.4}
\begin{adjustbox}{width=\columnwidth,center,scale=0.85}
\begin{tabular}{|c|c|c|}
	\hline 
	$\square = \ulZ$ &
	$\dbox = \ulZ^*$ & 
         $\circ=\ul B(3,0) $ 
         \\ 
         \hline
         \begin{tikzcd}[bend right=10, swap]
         & \Z \ar[dl,"1"] \ar[d,"1"] \ar[dr,"1"] & \\
         \Z \ar[dr,"1"] \ar[ur,"2", color=\inductioncolor]  & \Z \ar[d,"1"] \ar[u,"2", color=\inductioncolor]  & \Z \ar[ul,"2", color=\inductioncolor]  \ar[dl,"1"] \\
          & \Z \ar[d,"1"] \ar[ul,"2", color=\inductioncolor] \ar[u,"2", color=\inductioncolor] \ar[ur,"2", color=\inductioncolor] \\
          & \Z \ar[u,"2",color=\inductioncolor]
         \end{tikzcd}
         & 
         \begin{tikzcd}[bend right=10, swap]
         & \Z \ar[dl,"2"] \ar[d,"2"] \ar[dr,"2"] & \\
         \Z \ar[dr,"2"] \ar[ur,"1", color=\inductioncolor]  & \Z \ar[d,"2"] \ar[u,"1", color=\inductioncolor]  & \Z \ar[ul,"1", color=\inductioncolor]  \ar[dl,"2"] \\
          & \Z \ar[d,"2"] \ar[ul,"1", color=\inductioncolor] \ar[u,"1", color=\inductioncolor] \ar[ur,"1", color=\inductioncolor] \\
          & \Z \ar[u,"1", color=\inductioncolor]
         \end{tikzcd}
         &
         \begin{tikzcd}[bend right=10, swap]
         & \Z/8 \ar[dl,"1"] \ar[d,"1"] \ar[dr,"1"] & \\
         \Z/4 \ar[dr,"1"] \ar[ur,"2",color=\inductioncolor]  & \Z/4 \ar[d,"1"] \ar[u,"2",color=\inductioncolor]  & \Z/4 \ar[ul,"2"]  \ar[dl,"1"] \\
          & \Z/2  \ar[ul,"2",color=\inductioncolor] \ar[u,"2",color=\inductioncolor] \ar[ur,"2",color=\inductioncolor] \\
          & 0 
         \end{tikzcd}
         \\ 
         \hline
         $\bullet = \phi_{Q_8}^* \F_2 = \infl{\F_2}$
        &
        $ \raisebox{-0.4ex}{$\begin{tikzpicture}\node[draw,fill=white,circle,inner sep=1.25pt, text=black] (0,0) {\tiny $n$};  \end{tikzpicture}$}
        = \infl{\F_2}^n$
        &
        $ \raisebox{0.25ex}{$\phiLDRf$} = \phi_{IJK}^* \ulf$
         \\ 
         \hline
         \begin{tikzcd}[bend right=10, swap]
         & \F_2 & 
         \\
        0 & 0  & 0 
        \\
          &
          0 \\
          &
          0
         \end{tikzcd}
         &
         \begin{tikzcd}[bend right=10, swap]
         & \F_2^n & 
         \\
        0 & 0  & 0 
        \\
          &
          0 \\
          &
          0
         \end{tikzcd}
         &
         \begin{tikzcd}[bend right=10, swap]
         & 0 & \\
	\F_2  & \F_2 & \F_2 \\
	& 0 \\
          & 0
         \end{tikzcd}
         \\ 
         \hline
	$\fillpent=\phi_{IJK}^* \ulF$
        & 
	$\fillpentdual=\phi_{IJK}^* \ulF^*$
        & 
         $ \raisebox{-0.3ex}{$\phiZM$} = \phi_Z^*(\ul B(2,0))$
         \\ 
         \hline
          \begin{tikzcd}[swap]
         & \F_2^3 \ar[dl,"p_1"] \ar[d,"p_2"] \ar[dr,"p_3"'] & \\
         \F_2  & \F_2   & \F_2   \\
          & 0  \\
          & 0 
         \end{tikzcd}
          &
          \begin{tikzcd}
         & \F_2^3  & \\
         \F_2 \ar[ur,"i_1",color=\inductioncolor]  & \F_2  \ar[u,"i_2",color=\inductioncolor]  & \F_2  \ar[ul,"i_3"',color=\inductioncolor]  \\
          & 0  \\
          & 0 
         \end{tikzcd}
         &
          \begin{tikzcd}[bend right=10, swap]
         & \Z/4 \ar[dl,->>] \ar[d,->>] \ar[dr,->>] & \\
         \Z/2  \ar[ur,"2",color=\inductioncolor]  & \Z/2 \ar[u,"2",color=\inductioncolor] & \Z/2 \ar[ul,"2",color=\inductioncolor]  \\
          & 0   \\
          & 0 
         \end{tikzcd}
         \\ 
         \hline
         $\filltrap = \ul{mg}$
         &
         $\filltrapdual = \ul{mg}^*$ 
         &
         $ \raisebox{-0.2ex}{$\mgwsymbol$} = \mgw$
         \\ 
         \hline
          \begin{tikzcd}[swap]
         & \F_2^2 \ar[dl,"p_1"] \ar[d,"+"] \ar[dr,"p_2"'] & \\
         \F_2  & \F_2   & \F_2   \\
          & 0  \\
          & 0 
         \end{tikzcd}
          &
          \begin{tikzcd}
         & \F_2^2  & \\
         \F_2 \ar[ur,"i_1",color=\inductioncolor]  & \F_2  \ar[u,"\Delta",color=\inductioncolor]  & \F_2  \ar[ul,"i_2"',color=\inductioncolor]  \\
          & 0  \\
          & 0 
         \end{tikzcd}
         &
        \begin{tikzcd}[bend right=10, swap]
         & \F_2^2 \ar[dl, "2p_1"] \ar[d, "2\nabla"] \ar[dr, "2p_2" pos=0.75] & \\
         \Z^\sigma\!/4 \ar[dr,->>] \ar[ur, "i_2 q"  pos=0.25 ,color=\inductioncolor]  & \Z^\sigma\!/4 \ar[d,->>] \ar[u, "\Delta q",color=\inductioncolor]  & \Z^\sigma\!/4 \ar[ul, "i_1 q",color=\inductioncolor]  \ar[dl, ->>] \\
          & \Z/2  \ar[ul,"2",color=\inductioncolor] \ar[u,"2",color=\inductioncolor] \ar[ur,"2",color=\inductioncolor] \\
          & 0 
         \end{tikzcd}
         \\
         \hline
\end{tabular} 
\end{adjustbox}
}
\end{table}

\clearpage

\section{Charts}

We include below charts for $\upi_{\blacklozenge}\HQ\ulZ$ and $\upi_{\blacklozenge}\HQN$, separated into the positive and negative cones. Shading indicates where the augmentation $N_{C_2}^{Q_8} \ulZ \to \ulZ$ induces an isomorphism. The Mackey functors are as displayed in \zcref{tab-Q8Mackey}.

\begin{figure}[H]
\caption{The positive cone of $\upi_{\blacklozenge} \HQ  \ulZ$}
\label{fig:posconeZ}
\includegraphics[width=0.8\textwidth]{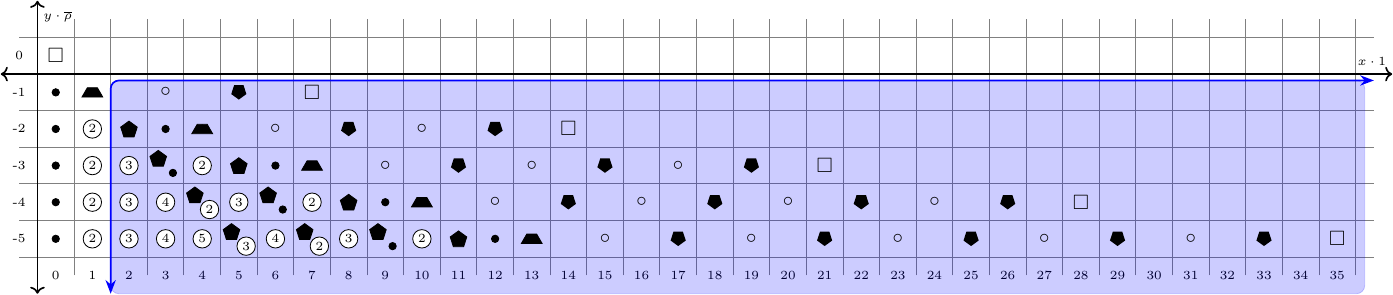}
\end{figure}

\begin{figure}[H]
\caption{The positive cone of $\upi_{\blacklozenge} \HQ N_{C_2}^{Q_8} \ulZ$}
\label{fig:posconeN}
\includegraphics[width=0.8\textwidth]{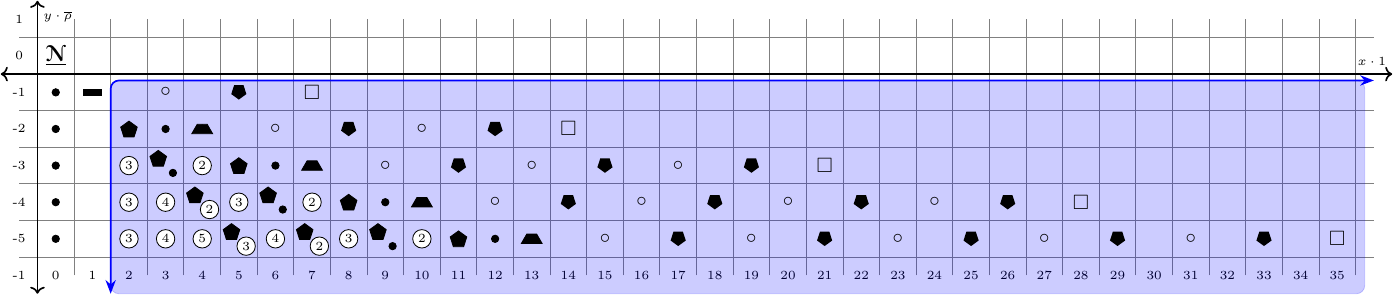}
\end{figure}

\begin{figure}[H]
\caption{The negative cone of $\upi_{\blacklozenge} \HQ \ulZ$}
\label{fig:negconeZ}
\includegraphics[width=0.8\textwidth]{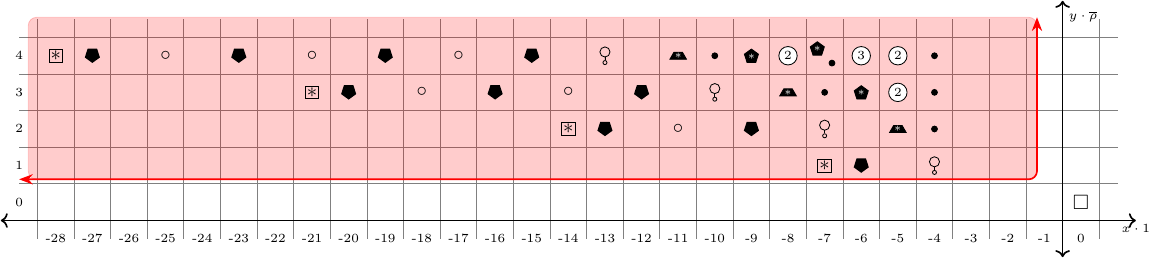}
\end{figure}

\begin{figure}[H]
\caption{The negative cone of $\upi_{\blacklozenge} \HQ N_{C_2}^{Q_8} \ulZ$}
\label{fig:negconeN}
\includegraphics[width=0.8\textwidth]{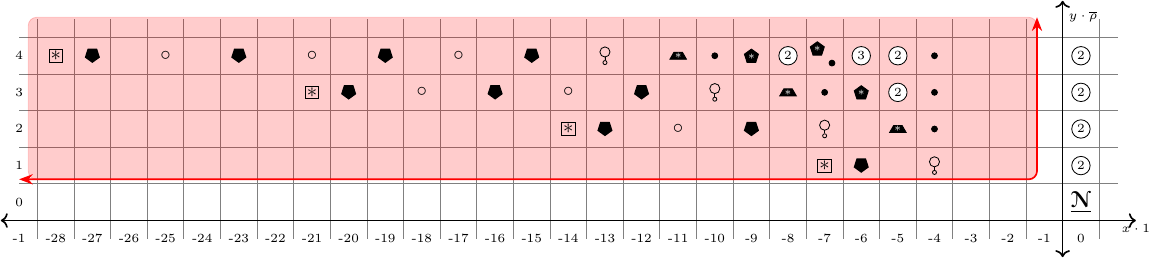}
\end{figure}

\clearpage

\bibliographystyle{amsalpha}

\end{document}